\documentclass[reqno,centertags,12pt]{amsart}

\usepackage{amssymb,amsmath,amsthm,dsfont}
\usepackage[T1]{fontenc}
\usepackage{mathrsfs}

\theoremstyle{plain} \newtheorem{theorem}{Theorem}[section]
\newtheorem{lemma}[theorem]{Lemma}
\newtheorem{coro}[theorem]{Corollary}
\newtheorem{prop}[theorem]{Proposition}
\theoremstyle{definition} 
\theoremstyle{remark} \newtheorem{remark}{Remark}

\newcommand{\R}{{\mathbb R}}
\newcommand{\C}{{\mathbb C}}
\newcommand{\Z}{{\mathbb Z}}
\newcommand{\N}{{\mathbb N}}
\newcommand{\T}{{\mathbb T}}
\newcommand{\iT}{\T_\alpha}
\newcommand{\sC}{{\mathcal C}}
\newcommand{\sR}{{\mathcal R}}
\newcommand{\sS}{{\mathcal S}}

\newcommand{\sB}{{\mathcal B}}
\newcommand{\sO}{{\mathcal O}}
\newcommand{\fS}{{\mathfrak S}}
\newcommand{\scC}{{\mathscr C}}
\newcommand{\scR}{{\mathscr R}}
\DeclareMathOperator{\supp}{supp}

\begin{document}
\title[Strichartz estimates on irrational tori]{Strichartz estimates for Schr\"odinger equations on irrational tori in two and three dimensions}
\author[N.~Strunk]{Nils~Strunk}

\thanks{The author was supported by the German Research Foundation, Collaborative Research Center 701.}

\address{Universit\"at Bielefeld, Fakult\"at f\"ur Mathematik, Postfach 100131, 33501 Bielefeld, Germany}
\email{strunk@math.uni-bielefeld.de}

\begin{abstract}
  In this paper we prove new multilinear Strichartz estimates, which are obtained by using techniques
  of Bourgain.
  These estimates lead to new critical well-posedness results for the nonlinear Schr\"odinger equation
  on irrational tori in two and three dimensions with small initial data. In three dimensions, this includes the
  energy critical case. This extends recent work of Guo--Oh--Wang.
\end{abstract}

\subjclass[2010]{Primary 35Q55; Secondary 35B33}
\keywords{}

\maketitle

\section{Introduction}\label{sect:intro}
\noindent
The aim of this paper is to obtain new critical well-posedness results for the nonlinear Schr\"odinger equation
(NLS) posed on two-dimensional and three-dimensional irrational tori. For $d,k\in\N$ we consider the following
Cauchy problem of the NLS
\begin{equation}\label{eq:nls_irr}
   \left\{
     \begin{array}{rcll}
       i\partial_t u-\Delta u &=& \pm|u|^{2k}u  \\
       u|_{t=0}                &=& \phi \in H^s(\iT^d),
     \end{array}
   \right.
\end{equation}
where we follow the notation of \cite{GOW13} and denote the flat irrational torus by
\begin{equation}\label{eq:irr_torus}
   \iT^d = \prod_{j=1}^d \R/(\alpha_j\T),\quad \frac1C<\alpha_j<C,\;j=1,\ldots,d.
\end{equation}
By a change of spatial variables, \eqref{eq:nls_irr} is equivalent to the following nonlinear
Schr\"odinger equation on the rational torus $\T^d=\R^d/\Z^d$:
\begin{equation}\label{eq:nls_rat}
   \left\{
     \begin{array}{rcll}
       i\partial_t u-\Delta u &=& \pm|u|^{2k}u \\
       u|_{t=0}                &=& \phi \in H^s(\T^d),
     \end{array}
   \right.
\end{equation}
where the Laplace operator $\Delta$ is defined via
\[
   \widehat{\Delta f}(n) = -4\pi^2 Q(n)\widehat f(n),
\]
with $n=(n_1,\ldots,n_d)\in\Z^d$ and $Q(n)=\alpha_1n_1^2+\ldots+\alpha_dn_d^2$.
In the present paper, we study \eqref{eq:nls_rat}.

\medskip

The (scaling-)crit\-i\-cal Sobolev index is given by
\begin{equation}\label{eq:critical_index}
   s_c=\frac{d}{2}-\frac1k.
\end{equation}
For strong solutions $u\colon(-T,T)\times\T^d\to\C$ one easily verifies that the energy is conserved, i.e.\
\begin{equation}\label{eq:conserve_energy}
   E\bigl(u(t)\bigr) = \frac12 \int_{\T^d} |\nabla u(t,x)|^2\,dx
     \pm \frac{1}{2k+2} \int_{\T^d} |u(t,x)|^{2k+2}\,dx = E(\phi),
\end{equation}
and so is the $L^2$-mass, i.e.\
\begin{equation}\label{eq:conserve_l2}
   M\bigl(u(t)\bigr) = \frac12 \int_{\T^d} |u(t,x)|^2\,dx = M(\phi).
\end{equation}
Hence, the problem with $k=2$ in dimension $d=3$ is called energy-critical.

\medskip

Several authors studied critical well-posedness for NLS on flat rational tori \cite{HTT11,HTT13,W13a,W13b,GOW13}.
However, to our knowledge the first critical well-posedness results for the NLS on irrational tori have been 
recently obtained by Guo--Oh--Wang \cite[Theorem~1.7]{GOW13}. They proved critical well-posedness for small
data in the cases
\[
   \text{(i)\;\,$d=2$: }  k\geq 6, \qquad
   \text{(i\/i)\;\,$d=3$: }  k\geq 3,\qquad
   \text{(i\/i\/i)\;\,$d\geq4$: }  k\geq 2.
\]
Furthermore, they considered the energy critical case $d=3$ and $k=2$ on \emph{partially} irrational tori,
where two periods are the same \cite[Appendix~B]{GOW13}.
In the present paper we are going to extend the known results in two and three dimensions to
\begin{equation}\label{eq:d_k_cond}
   \text{(i)\;\,$d=2$: }  k\geq3,\qquad\qquad
   \text{(i\/i)\;\,$d=3$: } k=2.
\end{equation}

\medskip

We use an approach similar to  Bourgain's linear Strichartz estimate for
three-dimensional irrational tori \cite{B07}. Guo--Oh--Wang applied this idea to obtain well-posedness
in several dimensions \cite{GOW13}.
However, in contrast to \cite{GOW13}, we use mixed $L^p((-T,T),L^q(\T^d))$ spaces to
improve the trilinear Strichartz estimate \cite[Proposition~5.7]{GOW13}, leading to the corresponding result in
the energy critical case $d=3$ and $k=2$. In addition to that, we use a refined trilinear Strichartz estimate in
two dimensions (Lemma~\ref{lem:trilinear_est_2d}) to treat the two dimensional case.

\medskip

In this paper, we will focus on the multilinear Strichartz estimates in Proposition~\ref{prop:trilinear_est_2d} and
Proposition~\ref{prop:trilinear_est}.
These propositions serve as a replacement for Proposition~5.7 in \cite{GOW13}, which implies
critical well-posedness results by standard arguments, cf.\ \cite[Section~5]{GOW13}, \cite[Section~4]{HTT11}, and
the references therein:
Define appropriate iteration spaces that go
back to Herr--Tataru--Tzvetkov \cite[Definition~2.6--2.7]{HTT11}, in which one may control the Duhamel term,
cf.\ \cite[Proposition~4.7]{HTT11} and \cite[Proposition~5.6]{GOW13}. Finally, a fixed-point argument proves
local well-posedness \cite[Proof of Theorem~1.1]{HTT11}. Global well-posedness for small data follows
essentially from the conservation laws \eqref{eq:conserve_energy}, \eqref{eq:conserve_l2},
see e.g.\ \cite[Proof of Theorem~1.2]{HTT11}. Hence our results lead to:

\begin{theorem}\label{thm:wp}
  Let $s_c$ be defined by \eqref{eq:critical_index} and let $d,k\in\N$ satisfy \eqref{eq:d_k_cond}.
  Then, for all $s\geq s_c$ the initial value problem \eqref{eq:nls_irr} is locally well-posed in $H^s(\iT^d)$, and
  globally well-posed in $H^s(\iT^d)$, if the initial value is small in $H^{s_c}(\iT^d)$.
\end{theorem}
A more precise formulation of the theorem may be found for instance in \cite[Theorem~1.1 and Theorem~1.2]{HTT11}.

\medskip

The paper is organized as follows: In Section~\ref{sect:notation} we introduce some basic notation that will be used
later on.
The two and three dimensional cases are considered in Section~\ref{sect:2d} and Section~\ref{sect:3d}, respectively.
The paper is not self-contained. We rely on results from \cite{B89,B07,GOW13,H12}.

\section{Notation}\label{sect:notation}
\noindent
The following notations are quite standard, see e.g.\ \cite{HTT11}.
We will write $A\lesssim B$, if there exists a harmless constant $c>0$ such that $A\leq c B$.
Define the spatial Fourier coefficients 
\[
  \widehat f(n) := \int_{[0,1]^d} e^{-2\pi ix\cdot n} f(x)\,dx,\quad n\in\Z^d.
\]

We fix a non-negative, even function $\psi\in C^\infty_0((-2,2))$ with $\psi(s)=1$ for $|s|\leq1$ to define
a partition of unity: for a dyadic number $N\geq 1$, we set
\begin{equation}\label{eq:def_cutoff}
   \psi_N(\xi) := \psi\biggl(\frac{|\xi|}{N}\biggr) - \psi\biggl(\frac{2|\xi|}{N}\biggr)
   \quad \text{for }N\geq 2,\quad \psi_1(\xi) := \psi(|\xi|).
\end{equation}
We also define the frequency localization operators $P_N\colon L^2(\T^d)\to L^2(\T^d)$ as the Fourier multiplier
with symbol $\psi_N$. Moreover, we define $P_{\leq N}:=\sum_{1\leq M\leq N} P_M$.

More generally, given a set $\sS\subseteq\Z^d$, we define $P_\sS$ to be the Fourier multiplier operator with
symbol $\chi_\sS$, where $\chi_\sS$ denotes the characteristic function of $\sS$.

For $N,M\geq 1$ we define the collection of rectangular sets
\begin{multline*}
   \scR_{N,M}:=\bigl\{\sC\subseteq\Z^d : \exists \text{$z\in\Z^d$, $O$ orthogonal $d\times d$-Matrix s.t.}\\
     O\sC+z\subseteq [-N,N]^{d-1}\times[-M,M] \bigr\}.
\end{multline*}
Furthermore, we set $\scC_N:=\scR_{N,N}$.

For $s\in\R$, we define the Sobolev space $H^s(\T^d)=(1-\Delta)^{-\frac{s}{2}}L^2(\T^d)$ endowed with the norm
\[
   \|f\|^2_{H^s(\T^d)} = \sum_{n\in\Z^d} \langle n\rangle^{2s} |\widehat f(n)|^2_{L^2(\T^d)},
   \quad\text{where }\langle x\rangle = (1+|x|^2)^{\frac12}.
\]
We denote the linear Schr\"odinger evolution by
\[
   (e^{it\Delta} f)(x) = \sum_{n\in\Z^d} \widehat f(n) e^{2\pi i(n\cdot x + Q(n)t)}.
\]

\section{The two dimensional case}\label{sect:2d}
\noindent
The following lemma may be seen as a variant of the Hausdorff--Young inequality and was used by several authors,
e.g.\ \cite{B07,GOW13}. This lemma is one of our most important ingredients, hence, we are going to prove it.
\begin{lemma}\label{lem:point_est}
  Let $p\geq 2$ and $r\geq1$. There exists a compact interval $I\subseteq\R$ such that for all bounded
  $\sS\subseteq \Z^d$, $f\colon\sS\to\R$, and $\fS_k$ defined by
\[
   \fS_k:=\bigl\{ n\in\sS : |f(n)-k|\leq r \bigr\}, \quad k\in\Z,
\]
  the following estimate holds true
\[
   \|\#\fS_k\|_{\ell_k^p} \lesssim \biggl\| \sum_{n\in\sS} e^{2\pi i f(n)t} \biggr\|_{L_t^{\frac{p}{p-1}}(I)}.
\]
\end{lemma}
\begin{proof}
  The proof may be found in \cite[(1.1.8')--(1.1.9)]{B07}. However, we are going to spell out some details.
  Let $\eta\colon\R\to\R$ be a continuous, compactly supported function with $\widehat\eta(\tau)\geq0$
  for all $\tau\in\R$ and $\widehat\eta(\tau)\geq 1$ for all $\tau\in[-r,r]$, e.g.\ for $c>0$ large enough
\[
   \eta(t)=c\,\chi_{[-r,r]}\ast\chi_{[-r,r]}(t),
\]
  with Fourier transform $\widehat\eta(\tau)=c\bigl(\frac{\sin (2\pi r \tau) }{\tau}\bigr)^2$.
  Define $\psi\colon \R\to\R$ by
\[
   \psi(t) = \sum_{n\in\sS} e^{2\pi if(n)t} \eta(t),
\]
  and set $I:=\supp\psi$. Then, the Hausdorff--Young inequality implies
\[
   \|\#\fS_k\|_{\ell_k^p} \leq \biggl\| \sum_{n\in\sS} \widehat\eta\bigl(k-f(n)\bigr)\biggr\|_{\ell_k^p}
   = \bigl\| \widehat\psi(k) \bigr\|_{\ell_k^p}
   \lesssim \biggl\|\sum_{n\in\sS} e^{2\pi if(n)t} \biggr\|_{L_t^{\frac{p}{p-1}}(I)}.
\]
\end{proof}

In several applications it turned out to be beneficial to use almost orthogonality in time.
This was first observed by Herr--Tataru--Tzvetkov \cite[Proof of Proposition~3.5]{HTT11} for the rational torus
$\T^3$, and later also applied for Zoll manifolds such as $\mathbb S^3$ \cite[Proof of Proposition~3.6]{H12}.
The next lemma shows what this means in our setting.
\begin{lemma}\label{lem:orth}
  Let $\sigma>0$ and $\tau_0\subset\R$ be a bounded interval. Furthermore, let $\tau_1\supset\overline{\tau_0}$
  be an open interval. Then, for all $\phi_1,\ldots,\phi_{2k+1}\in L^2(\T^d)$ and dyadic numbers
  $N_1\geq\ldots\geq N_{2k+1}\geq 1$ there exist rectangles $\sR_\ell\in\scR_{N_2,M}$, where
  $M=\max\bigl\{\frac{N_2^2}{N_1},1\bigr\}$, such that $P_{N_1}=\sum_{\ell}P_{\sR_\ell}P_{N_1}$ and
\begin{multline*}
  \biggl\|\prod_{j=1}^{2k+1} P_{N_j}e^{it\Delta} \phi_j \biggr\|_{L^2(\tau_0\times\T^d)}^2
  \lesssim \sum_{\ell} \biggl\|P_{\sR_\ell}P_{N_1} e^{it\Delta}\phi_1
     \prod_{j=2}^{2k+1}P_{N_j}e^{it\Delta}\phi_j\biggr\|_{L^2(\tau_1\times\T^d)}^2\\
    + N_2^{-\sigma}\prod_{j=1}^{2k+1} \|P_{N_j}\phi_j\|_{L^2(\T^d)}^2.
\end{multline*}
\end{lemma}
\begin{proof}
  Thanks to spatial orthogonality, we may replace $P_{N_1}e^{it\Delta}\phi_1$ by $P_{\sC}e^{it\Delta}\phi_1$, where
  $\sC\in\scC_{N_2}$. Without loss of generality, we may assume $N_1\gg N_2$.
  As in the proof of \cite[Proposition~3.5]{HTT11}, we define the following partition:
  Let $\xi_0\in\Z^d$ be the center of $\sC$. We define almost disjoint strips of width
  $M=\max\bigl\{\frac{N_2^2}{N_1},1\bigr\}$, which are orthogonal to $\xi_0$:
\[
   \sR_\ell:=\bigl\{ \xi\in\sC : \xi\cdot\xi_0\in\bigl[|\xi_0|M\ell, |\xi_0|M(\ell+1)\bigr) \bigr\}\in\scR_{N_2,M},
   \quad \N\ni\ell\approx\frac{N_1}{M}.
\]
  Since $\sC=\bigcup_\ell\sR_\ell$, we have $P_{\sC}P_{N_1}e^{it\Delta}\phi_1=\sum_\ell P_{\sR_\ell}P_{N_1}e^{it\Delta}\phi_1$.

  Let $\eta\in C^\infty_0(\tau_1)$ be a cutoff function satisfying $\eta(t)=1$ for all $t\in\tau_0$. Obviously,
\[
   \biggl\|\prod_{j=1}^{2k+1}P_{N_j}e^{it\Delta}\phi_j\biggr\|_{L^2(\tau_0\times\T^d)}^2
   \leq \biggl\|\eta(t)\prod_{j=1}^{2k+1}P_{N_j}e^{it\Delta}\phi_j\biggr\|_{L^2(\R\times\T^d)}^2\lesssim I_1+I_2,
\]
  where
\[
   I_1:=\sum_{\ell\approx N_1/M} \biggl\|P_{\sR_\ell}P_{N_1} e^{it\Delta}\phi_1
     \prod_{j=2}^{2k+1}P_{N_j}e^{it\Delta}\phi_j\biggr\|_{L^2(\tau_1\times\T^d)}^2,
\]
  and
\begin{multline*}
   I_2:=\sum_{\substack{\ell,\ell'\approx N_1/M:\\|\ell-\ell'|\gg1}}
       \biggl\langle \eta(t)^2 P_{\sR_\ell}P_{N_1} e^{it\Delta}\phi_1 \prod_{j=2}^{2k+1}P_{N_j}e^{it\Delta}\phi_j,\\
        P_{\sR_{\ell'}}P_{N_1} e^{it\Delta}\phi_1 \prod_{j=2}^{2k+1}P_{N_j}e^{it\Delta}\phi_j \biggr\rangle_{L^2(\R\times\T^d)}.
\end{multline*}
  We have to show that $I_2\lesssim N_2^{-\sigma}\prod_{j=1}^{2k+1} \|P_{N_j}\phi_j\|_{L^2(\T^d)}^2 $.
  Interpreting the integration with respect to $t$ as Fourier transform on $\R$ and taking the absolute value,
  we get
\begin{multline*}
   |I_2| \lesssim  \sum_{\substack{\ell,\ell'\approx N_1/M:\\|\ell-\ell'|\gg1}}
     \sum_{\substack{n_1\in\sR_\ell,\;n_1'\in\sR_{\ell'},\\n_j,n_j'\in \sB_{N_j},\;j=2\ldots 2k+1}}
     \bigl|\widehat{\eta^2}\bigr|\biggl(\sum_{j=1}^{2k+1}Q(n_j')-Q(n_j)\biggr)\\
     \times\prod_{j=1}^{2k+1}\bigl|\widehat{\phi_j}(n_j)\bigr|\bigl|\overline{\widehat{\phi_j}(n_j')}\bigr|,
\end{multline*}
  where $\sB_N:=\supp \psi_N$ with $\psi_N$ defined in \eqref{eq:def_cutoff}.
  Similar to the proof of \cite[Proposition~3.5]{HTT11} we get
\[
   \biggl|\sum_{j=1}^{2k+1}Q(n_j')-Q(n_j)\biggr| = M^2|\ell-\ell'|(\ell+\ell')+\sO(M^2\ell)+\sO(N_2^2)
   \gtrsim N_2^2\langle \ell-\ell' \rangle,
\]
  since $\ell,\ell'\approx\frac{N_1}{M}$ and $|\ell-\ell'|\gg1$.
  Thus, for any $\beta>0$, we may estimate
\[
   \bigl|\widehat{\eta^2}\bigr|\biggl(\sum_{j=1}^{2k+1}Q(n_j')-Q(n_j)\biggr)
   \lesssim_\beta N_2^{-2\beta}\langle \ell-\ell' \rangle^{-\beta}.
\]
  Using Cauchy-Schwarz with respect to $n_j$, $n_j'$, $j=1,\ldots,2k+1$, yields
\[
   |I_2| \lesssim N_2^{-\sigma}
     \sum_{\substack{\ell,\ell'\approx N_1/M:\\|\ell-\ell'|\gg1}}\langle \ell-\ell' \rangle^{-\beta}
     \|P_{\sR_\ell}\phi_j\|_{L^2}\|P_{\sR_{\ell'}}\phi_j\|_{L^2}
     \prod_{j=2}^{2k+1} \|P_{N_j}\phi_j\|_{L^2}^2,
\]
  where $\sigma = 2\beta-(2k+1)d$ and $L^2=L^2(\T^d)$.
  Finally, Schur's lemma implies
\[
   \sum_{\substack{\ell,\ell'\approx N_1/M:\\|\ell-\ell'|\gg1}}\langle \ell-\ell' \rangle^{-\beta}
     \|P_{\sR_\ell}\phi_j\|_{L^2}\|P_{\sR_{\ell'}}\phi_j\|_{L^2}
   \lesssim \|P_{N_1}\phi_1\|_{L^2(\T^d)}^2,
\]
  provided $\beta>1$.
\end{proof}

In the following let $\tau_0\subseteq[0,1]$ be any time interval and $k\geq3$.
We are going to prove the following proposition:
\begin{prop}\label{prop:trilinear_est_2d}
  There exists $\delta>0$ such that for all $\phi_1,\ldots,\phi_{k+1}\in L^2(\T^2)$ and dyadic
  numbers $N_1\geq\ldots\geq N_{k+1}\geq 1$ the following estimate holds true
\begin{multline*}
  \biggl\|\prod_{j=1}^{k+1} P_{N_j}e^{it\Delta} \phi_j \biggr\|_{L^2(\tau_0\times\T^2)}\\
  \lesssim \biggl(\frac{N_{k+1}}{N_1}+\frac{1}{N_2}\biggr)^{\delta} \|P_{N_1}\phi_1\|_{L^2(\T^2)}
    \prod_{j=2}^{k+1} N_j^{s_c} \|P_{N_j}\phi_j\|_{L^2(\T^2)}.
\end{multline*}
\end{prop}

Before we start proving Proposition~\ref{prop:trilinear_est_2d}, we turn to the following
trilinear Strichartz estimate, which improves \cite[Lemma~5.9]{GOW13} for $d=2$, using ideas of \cite{B07}.
Combined with Lemma~\ref{lem:orth}, this is the essence of the proof of Proposition~\ref{prop:trilinear_est_2d}.
Note that the implicit constants depend on $C$ in \eqref{eq:irr_torus} and the local time interval $\tau_0$.
\begin{lemma}\label{lem:trilinear_est_2d}
  Let $2<p\leq4$. Then, for all $N,M\geq 1$ with $N\geq M$, $\sC_1\in\scC_N$,
  $\sC_2,\sC_3\in\scC_M$, and $\phi_1,\phi_2,\phi_3\in L^2(\T^2)$ we have
\[
   \biggl\|\prod_{j=1}^3P_{\sC_j} e^{it\Delta} \phi_j \biggr\|_{L^p(\tau_0,L^2(\T^2))}
   \lesssim M^{2-\frac2p} \prod_{j=1}^3\|P_{\sC_j}\phi\|_{L^2(\T^2)}.
\]
\end{lemma}
\begin{proof}
  This proof is a trilinear variant of the poof of \cite[Proposition~1.1]{B07}. Hence,
  we omit some details and refer the reader also to e.g.\ \cite[Proof of Proposition~2.2]{GOW13}.
  We will write $L_t^pL^q_x:=L^p(\tau_0,L^q(\T^2))$ and $L_t^p:=L^p(\tau_0)$ for brevity. 

  The left hand side may be estimated by
\begin{equation*}
  \biggl[ \sum_{a\in\Z^2} \biggl\| \sum_{\substack{n\in\sC_2,\\m\in\sC_3}} \widehat \phi_1(a-n-m) \widehat \phi_2(n)
          \widehat \phi_3(m) e^{2\pi i(Q(a-n-m)+Q(n)+Q(m))t} \biggr\|_{L^p_t}^2 \biggr]^{\frac12}
\end{equation*}
  using Plancherel's identity with respect to $x$ and Minkowski's inequality.
  Applying Hausdorff--Young (cf.\ \cite[Lemma~2.1]{GOW13}) and setting $c_{j,n}:=|\widehat \phi_j(n)|$ yields
\[
   \|\cdots\|_{L^p_t}
   \lesssim \biggl[ \sum_{k\in\Z} \biggl| \sum_{|Q(a-n-m)+Q(n)+Q(m)-k|\leq \frac12} c_{1,a-n-m}c_{2,n}c_{3,m}
            \biggr|^{\frac{p}{p-1}} \biggr]^{\frac{p-1}{p}}.
\]
  One easily verifies that $|Q(a-n-m)+Q(n)+Q(m)-k|\leq \frac12$ may be written as
\[
   |Q(3\widetilde n - 2a) + 3Q(\widetilde m) + 2Q(a)-6k| \leq 3,
\]
  where $\widetilde n:=n+m$ and $\widetilde m:=n-m$. Hence,
\[
   \#\bigl\{(n,m)\in\sC_2\times\sC_3 : |Q(a-n-m)+Q(n)+Q(m)-k|\leq \tfrac12\bigr\}
   \lesssim \#\fS_{\ell},
\]
  with
\[
   \fS_{\ell} := \bigl\{(\widetilde n,\widetilde m)\in\widetilde\sC_2\times\widetilde\sC_3 :
        |Q(\widetilde n) + 3Q(\widetilde m) - \ell|\leq4\bigr\},
\]
  $\ell:=\lfloor 6k-2Q(a)\rfloor\in\Z$, and cubes $\widetilde\sC_2,\widetilde\sC_3\in\scC_{3M}$.
  This conclusion and two applications of H\"older's estimate yield (cf.\ \cite[(1.1.5)--(1.1.7)]{B07})
\[
   \|\cdots\|_{L^p_t} \lesssim
   \biggl( \sum_{\ell\in\Z} (\#\fS_\ell)^{\frac{p}{p-2}} \biggr)^{\frac{p-2}{2p}}
   \biggl( \sum_{\substack{n\in\sC_2,\\m\in\sC_3}} c_{1,a-n-m}^2c_{2,n}^2c_{3,m}^2 \biggr)^{\frac12},
\]
  thus we arrive at
\[
   \biggl\|\prod_{j=1}^3P_{\sC_j} e^{it\Delta} \phi_j \biggr\|_{L^p(\tau_0,L^2(\T^2))}
   \lesssim \biggl( \sum_{\ell\in\Z}(\#\fS_\ell)^{\frac{p}{p-2}}\biggr)^{\frac{p-2}{2p}}
     \prod_{j=1}^3 \|P_{\sC_j}\phi_j\|_{L^2(\T^2)}.
\]
  The assumption $p\leq4$ ensures that $\frac{p}{p-2}\geq 2$ and hence, by Lemma~\ref{lem:point_est} we may estimate
\begin{align*}
   \biggl( \sum_{\ell\in\Z} (\#\fS_\ell)^{\frac{p}{p-2}} \biggr)^{\frac{p-2}{p}}
   & \lesssim \biggl\| \prod_{j=1}^2 \sum_{\varDelta \widetilde n_j,\varDelta \widetilde m_j\approx M} 
        e^{2\pi i\alpha_j (\widetilde n_j^2+ 3\widetilde m_j^2)t} \biggr\|_{L_t^\frac{p}{2}(I)} \\ 
   & \lesssim \biggl\| \sum_{\varDelta n\approx M} e^{2\pi in^2t} \biggr\|_{L^{2p}_t(I)}^4
     \lesssim M^{4(1-\frac{1}{p})}
\end{align*}
  for some compact interval $I\subseteq\R$, provided $p>2$. Here, we use the notation
  $\sum_{\varDelta n\approx N} = \sum_{n=a}^{a+cN}$
  for some $a\in\Z$ and harmless $c\in\N$. The last inequality goes back to Bourgain
  \cite[Proposition~1.10 and Section~4]{B89}. However, we need a slight
  variation here, which may be found in \cite[Lemma~3.1]{H12}. This implies the desired estimate.
\end{proof}
\begin{coro}\label{coro:linear_est_2d}
  Let $p>6$.
  \begin{enumerate}
  \item\label{it:2d_cube} For every $N\geq1$, $\sC\in\scC_N$ and $\phi\in L^2(\T^2)$ we have
\[
   \|P_\sC e^{it\Delta} \phi\|_{L^p(\tau_0,L^6(\T^2))} \lesssim N^{\frac23-\frac{2}{p}} \|P_\sC\phi\|_{L^2(\T^2)}.
\]
  \item\label{it:2d_rect} Let $6\leq q<p$. Then, for all $N,M\geq1$ with $N\geq M$, $\sR\in\scR_{N,M}$
    and $\phi\in L^2(\T^2)$ it holds that
\[
   \|P_{\sR}e^{it\Delta} \phi\|_{L^p(\tau_0,L^q(\T^2))}
   \lesssim N^{\frac12+\frac1q-\frac{2}{p}} M^{\frac12-\frac3q} \|P_\sR\phi\|_{L^2(\T^2)}.
\]
  \end{enumerate}
\end{coro}
\begin{proof}
  The first estimate is a direct consequence of Lemma~\ref{lem:trilinear_est_2d}, provided $p\leq12$. 
  Bernstein's inequality implies the result for $p=\infty$:
\[
   \|P_\sC e^{it\Delta} \phi\|_{L^\infty(\tau_0\times\T^2)} \lesssim N\|P_\sC\phi\|_{L^2(\T^2)}.
\]
  For $12<p<\infty$, the desired estimate follows from H\"older's estimate and the estimates for $p=12$ and
  $p=\infty$.

  The second statement follows from \ref{it:2d_cube}, the estimate
\[
   \|P_{\sR}e^{it\Delta}\phi\|_{L^\infty(\tau_0\times \T^2)}
   \leq \bigl(\#(\sR\cap\Z^2)\bigr)^{\frac12} \|P_{\sR}\phi\|_{L^2(\T^2)}
   \lesssim (NM)^{\frac12} \|P_{\sR}\phi\|_{L^2(\T^2)},
\]
  which may easily be obtained by applying Cauchy-Schwarz in Fourier space, and H\"older's inequality.
  The conclusion works as follows: Set $f(t,x):=P_{\sR}e^{it\Delta}\phi(x)$, $\varepsilon=\frac{6p}{q}-6>0$ and
  $\vartheta=\frac{6}{q}\leq1$. Then,
\[
   \||f|^\vartheta |f|^{1-\vartheta}\|_{L^p_tL^q_x}
   \leq \|f\|_{L^{6+\varepsilon}_tL^6_x}^\vartheta\|f\|_{L^\infty_{t,x}}^{1-\vartheta}
   \lesssim N^{\frac12+\frac1q-\frac{2}{p}} M^{\frac12-\frac3q} \|P_{\sR}\phi\|_{L^2(\T^2)}.
\]
\end{proof}

\begin{proof}[Proof of Proposition~\ref{prop:trilinear_est_2d}]
  Thanks to Lemma~\ref{lem:orth} it suffices to replace $P_{N_1} e^{it\Delta} \phi_1$ by 
  $P_\sR P_{N_1} e^{it\Delta} \phi_1$, where $\sR\in\scR_{N_2,M}$ with $M=\max\bigl\{\frac{N_2^2}{N_1},1\bigr\}$,
  provided we magnify the time interval to an open interval $\tau_1\supset\overline{\tau_0}$.

  Let $6<p_1,q_1<8$ and $3<p_2\leq\frac{24}{5}$. Furthermore, let $p_3$ and $q_2$ be defined via the relations
  $\frac12=\frac{1}{p_1}+\frac{1}{p_2}+\frac{k-2}{p_3}$ and $\frac12=\frac{1}{q_1}+\frac13+\frac{k-2}{q_2}$,
  respectively. By H\"older's estimate the following holds true:
\begin{multline*}
   \biggl\|P_\sR P_{N_1} e^{it\Delta} \phi_1 \prod_{j=2}^{k+1}P_{N_j} e^{it\Delta} \phi_j\biggr\|_{L^2_{t,x}}
   \leq   \|P_{\sR} P_{N_1}e^{it\Delta}\phi_1\|_{L^{p_1}_tL^{q_1}_x}\\
   \times \|P_{N_2}e^{it\Delta}\phi_2P_{N_3}e^{it\Delta}\phi_3\|_{L^{p_2}_tL^3_x}
          \prod_{j=4}^{k+1} \| P_{N_j}e^{it\Delta}\phi_j\|_{L^{p_3}_tL^{q_2}_x},
\end{multline*}
  where $L^r_tL^s_x:=L^r(\tau_1,L^s(\T^2))$ and $L^2_{t,x}:=L^2_tL^2_x$.
  Let $f_j:=P_{N_j}e^{it\Delta}\phi_j$, $j=2,3$, then we treat the bilinear term as follows:
\[
   \|f_2f_3\|_{L^{p_2}_tL^3_x}^2
   = \|f_2^2f_3^2\|_{L^{\frac{p_2}{2}}_t L^{\frac32}_x}
   \leq \|f_2f_3^2\|_{L^r_tL^2_x} \|f_2\|_{L^s_tL^6_x},
\]
  where $s>6$ and $\frac{2}{p_2}=\frac{1}{r}+\frac{1}{s}$. Note that $p_2\leq\frac{24}{5}$ ensures that
  $r\leq 4$. By Lemma~\ref{lem:trilinear_est_2d} and Corollary~\ref{coro:linear_est_2d}, we have for all $\eta>0$
\begin{equation}\label{eq:bilin_est_2d}
   \|P_{N_2}e^{it\Delta}\phi_2P_{N_3}e^{it\Delta}\phi_3\|_{L^{p_2}_tL^3_x} \leq N_2^{\frac16+\eta}N_3^{\frac76-\frac{2}{p_2}-\eta}.
\end{equation}
  Corollary~\ref{coro:linear_est_2d}, \eqref{eq:bilin_est_2d} and Sobolev's embedding imply
\begin{multline*}
   \biggl\|P_\sR P_{N_1} e^{it\Delta} \phi_1 \prod_{j=2}^{k+1}P_{N_j} e^{it\Delta} \phi_j\biggr\|_{L^2_{t,x}}
   \lesssim \Bigl(\frac{M}{N_2}\Bigr)^{\frac12-\frac{3}{q_1}} N_2^{\frac76-\frac{2}{p_1}-\frac{2}{q_1}+\eta} 
          N_3^{\frac76-\frac{2}{p_2}-\eta}\\
   \times \prod_{j=4}^{k+1}N_j^{1-\frac{2}{k-2}(\frac23-\frac{1}{p_1}-\frac{1}{p_2}-\frac{1}{q_1})}
          \|P_{\sR}P_{N_1}\phi_j\|_{L^2_x}\prod_{j=2}^{k+1} \|P_{N_j}\phi_j\|_{L^2_x},
\end{multline*}
  where $L^2_x:=L^2(\T^2)$.
  For all $0<\nu_1,\nu_2\ll1$, there exist $\delta>0$ and $p_1,q_1>6$ sufficiently close to $6$ as well as $p_2>3$
  sufficiently close to $3$ such that
\begin{flalign*}
   \;\; \text{(i)\;\,}   & \Bigl(\frac{M}{N_2}\Bigr)^{\frac12-\frac{3}{q_1}} = \Bigl(\frac{M}{N_2}\Bigr)^\delta &
      \text{(i\/i)\;\,} & N_2^{\frac76-\frac{2}{p_1}-\frac{2}{q_1}+\eta} = N_2^{\frac12+\nu_1+\eta} \;\;  \\
   \;\; \text{(i\/i\/i)\;\,} & N_3^{\frac76-\frac{2}{p_2}-\eta} = N_3^{\frac12+\nu_2-\eta} &
      \text{(i\/v)\;\,} & N_j^{1-\frac{2}{k-2}(\frac23-\frac{1}{p_1}-\frac{1}{p_2}-\frac{1}{q_1})}
                          = N_j^{1-\frac{\nu_1+\nu_2}{k-2}}, \;\;
\end{flalign*}
  where $j\in\{4,\ldots,k+1\}$.
  Since $\frac12<s_c<1$ for $k\geq3$, we may choose $0<\nu_1,\nu_2,\eta\ll 1$ small enough to get
\begin{multline*}
   \biggl\|P_\sR P_{N_1} e^{it\Delta} \phi_1 \prod_{j=2}^{k+1}P_{N_j} e^{it\Delta} \phi_j\biggr\|_{L^2_{t,x}}\\
   \lesssim 
      \Bigl(\frac{N_{k+1}}{N_1}+\frac{1}{N_2}\Bigr)^\delta
      \|P_{\sR}P_{N_1}\phi_1\|_{L^2_x}\prod_{j=2}^{k+1}N_j^{s_c}\|P_{N_j}\phi_j\|_{L^2_x}.
\end{multline*}
\end{proof}

\section{The three dimensional case}\label{sect:3d}
\noindent
Similar to the previous section, let $\tau_0\subseteq[0,1]$ be any time interval.
\begin{prop}\label{prop:trilinear_est}
  For all $\varepsilon>0$, there exists $\delta>0$ such that for all $\phi_1,\phi_2,\phi_3\in L^2(\T^3)$ and
  dyadic numbers $N_1\geq N_2\geq N_3\geq 1$ the following estimate holds true:
\[
  \biggl\|\prod_{j=1}^3P_{N_j}e^{it\Delta} \phi_j \biggr\|_{L^2(\tau_0\times\T^3)}
  \lesssim \biggl(\frac{N_3}{N_1}+\frac{1}{N_2}\biggr)^{\delta}
    N_2^{\frac34+\varepsilon} N_3^{\frac54-\varepsilon} \prod_{j=1}^3 \|P_{N_j}\phi_j\|_{L^2(\T^3)}.
\]
\end{prop}

We postpone the proof and recall a linear Strichartz estimate, which
is --- up to almost orthogonality --- the main ingredient for the trilinear Strichartz estimate in
Proposition~\ref{prop:trilinear_est}. This linear estimate goes back to Bourgain \cite{B07}. Again, note that the
implicit constants depend on $C$ in \eqref{eq:irr_torus} and the local time interval $\tau_0$.
\pagebreak[2]
\begin{lemma}\label{lem:linear_est}
  Let $p>\frac{16}{3}$.
  \begin{enumerate}
  \item\label{it:lin_bourgain} For every $N\geq 1$, $\sC\in\scC_N$ and $\phi\in L^2(\T^3)$ we have
\[
   \|P_\sC e^{it\Delta} \phi\|_{L^p(\tau_0,L^4(\T^3))} \lesssim N^{\frac34-\frac2p} \|P_\sC\phi\|_{L^2(\T^3)}.
\]
  \item Let $4\leq q<\frac{3p}{4}$. Then, for all $N,M\geq 1$ with $N\geq M$, $\sR\in\scR_{N,M}$ and all
    $\phi\in L^2(\T^3)$ it holds that
\[
   \|P_{\sR}e^{it\Delta} \phi\|_{L^p(\tau_0,L^q(\T^3))}
   \lesssim N^{1-\frac2p-\frac1q} M^{\frac12-\frac2q} \|P_\sR\phi\|_{L^2(\T^3)}.
\]
  \end{enumerate}
\end{lemma}
\begin{proof}[Proof/Reference]
  The first inequality was proved by Bourgain \cite[Prop\-o\-si\-tion~1.1]{B07}, see also \cite[Lemma~5.9]{GOW13}.

  The second estimate follows from \ref{it:lin_bourgain}, the estimate
\[
   \|P_{\sR}e^{it\Delta}\phi\|_{L^\infty(\tau_0\times \T^3)}
   \leq \bigl(\#(\sR\cap\Z^3)\bigr)^{\frac12} \|P_{\sR}\phi\|_{L^2(\T^3)}
   \lesssim NM^{\frac12} \|P_{\sR}\phi\|_{L^2(\T^3)},
\]
  and H\"older's estimate, see Corollary~\ref{coro:linear_est_2d}.
\end{proof}

\begin{proof}[Proof of Proposition~\ref{prop:trilinear_est}]
  From Lemma~\ref{lem:orth} we see that we may replace the projector $P_{N_1}$ by $P_{\sR}P_{N_1}$
  with $\sR\in\scR_{N_2,M}$ and $M=\max\bigl\{\frac{N_2^2}{N_1},1\bigr\}$, provided we enlarge the time interval
  $\tau_0$ to an open interval $\tau_1\supset\overline{\tau_0}$.
  
  Let $p_1>\frac{16}{3}$ and $4<q_1<\frac{3p_1}{4}$. Furthermore, let $p_2$ and $q_2$ be defined via the
  relations $\frac12=\frac{2}{p_1}+\frac{1}{p_2}$ and $\frac12=\frac{2}{q_1}+\frac{1}{q_2}$, respectively.
  H\"older's estimate yields
\begin{multline*}
   \|P_{\sR}P_{N_1}e^{it\Delta}\phi_1P_{N_2}e^{it\Delta}\phi_2P_{N_3}e^{it\Delta}\phi_3\|_{L^2_{t,x}}\\
   \leq  \|P_{\sR} P_{N_1}e^{it\Delta}\phi_1\|_{L^{p_1}_tL^{q_1}_x}
         \| P_{N_2}e^{it\Delta}\phi_2\|_{L^{p_1}_tL^{q_1}_x}
         \| P_{N_3}e^{it\Delta}\phi_3\|_{L^{p_2}_tL^{q_2}_x},
\end{multline*}
  where $L^r_tL^s_x:=L^r(\tau_1,L^s(\T^3))$ and $L^2_{t,x}:=L^2_tL^2_x$.
  Applying Lemma~\ref{lem:linear_est} and Sobolev's embedding we get
\begin{multline*}
   \|P_\sR P_{N_1}e^{it\Delta}\phi_1P_{N_2}e^{it\Delta}\phi_2P_{N_3}e^{it\Delta}\phi_3\|_{L^2_{t,x}} \\
   \begin{aligned}
     &\lesssim M^{\frac12-\frac{2}{q_1}} N_2^{\frac52-\frac{4}{p_1}-\frac{4}{q_1}} N_3^{\frac{4}{p_1}+\frac{6}{q_1}-1}
        \|P_\sR P_{N_1}\phi_j\|_{L^2_x}\prod_{j=2}^3 \|P_{N_j}\phi_j\|_{L^2_x}\\
     &\lesssim \biggl(\frac{N_3}{N_1}+\frac{1}{N_2}\biggr)^{\frac12-\frac{2}{q_1}} 
         N_2^{\frac72-\frac{4}{p_1}-\frac{8}{q_1}} N_3^{\frac{4}{p_1}+\frac{8}{q_1}-\frac32}
         \|P_\sR P_{N_1}\phi_j\|_{L^2_x}\prod_{j=2}^3 \|P_{N_j}\phi_j\|_{L^2_x},
   \end{aligned}
\end{multline*}
  where $L^2_x:=L^2(\T^3)$. The claim follows for $p_1$ sufficiently close to $\frac{16}{3}$ and $q_1$ sufficiently
  close to $4$.
\end{proof}

\begin{remark}
  Note that this proof can easily be modified to treat all $k\geq2$. However, since Guo--Oh--Wang already
  proved a similar statement for $k\geq 3$, we omitted this to make the argument more transparent.
\end{remark}

\bibliographystyle{amsplain}
\providecommand{\bysame}{\leavevmode\hbox to3em{\hrulefill}\thinspace}
\providecommand{\MR}{\relax\ifhmode\unskip\space\fi MR }
\providecommand{\MRhref}[2]{%
  \href{http://www.ams.org/mathscinet-getitem?mr=#1}{#2}
}
\providecommand{\href}[2]{#2}

\end{document}